\documentclass{amsart}
\usepackage{amssymb}
\usepackage{amsfonts}
\usepackage[all]{xy}
\usepackage{graphicx}

\newcommand{\sst}{\scriptscriptstyle}

\newcommand{\bA}{\mathbb{A}}
\newcommand{\fA}{\mathfrak{A}}
\newcommand{\C}{\mathbb{C}}

\newcommand{\D}{\mathbb{D}}

\newcommand{\bG}{\mathbb{G}}
\newcommand{\cN}{\mathcal{N}}

\newcommand{\R}{\mathbb{R}}
\newcommand{\Z}{\mathbb{Z}}

\newcommand{\fk}{\mathfrak{k}}
\newcommand{\fsl}{\mathfrak{sl}}

\newcommand{\ru}{{}^{\sst r}\!}

\newcommand{\ufl}{\mathchoice{{}^\flat}{{}^\flat}{{}^\flat}{\flat}}
\newcommand{\ush}{\mathchoice{{}^\sharp}{{}^\sharp}{{}^\sharp}{\sharp}}
\newcommand{\flatru}{{}^{\sst \flat r}}
\newcommand{\sharpru}{{}^{\sst \sharp r}}

\newcommand{\ffru}{{}^{\sst \flat(\flat r)}}
\newcommand{\sfru}{{}^{\sst \sharp(\flat r)}}

\newcommand{\q}{{}_{\sst q}}

\newcommand{\orb}[1]{\mathbb{O}(#1)}
\newcommand{\orbvb}[1]{\Omega(#1)}

\newcommand{\cC}{\mathcal{C}}
\newcommand{\cD}{\mathcal{D}}
\newcommand{\cF}{\mathcal{F}}
\newcommand{\cG}{\mathcal{G}}

\newcommand{\cL}{\mathcal{L}}
\newcommand{\cM}{\mathcal{M}}
\newcommand{\cO}{\mathcal{O}}
\newcommand{\cP}{\mathcal{P}}
\newcommand{\cQ}{\mathcal{Q}}

\newcommand{\uZ}{\underline{Z}}

\newcommand{\cIC}{\mathcal{IC}}

\newcommand{\cMn}{\cM^\natural}

\newcommand{\cg}[1]{\cC_G(#1)}
\newcommand{\cgl}[2]{\cC_G(#1)_{\le #2}}

\newcommand{\qg}[1]{\cQ_G(#1)}

\newcommand{\Db}[1]{\cD_{\sst G}^{\sst\mathrm{b}}(#1)}
\newcommand{\Dm}[1]{\cD_{\sst G}^{\sst -}(#1)}
\newcommand{\Dp}[1]{\cD_{\sst G}^{\sst +}(#1)}

\newcommand{\Dl}[2]{\Db{#1}^{\sst\le #2}}

\newcommand{\Dg}[2]{\Db{#1}^{\sst\ge #2}}

\newcommand{\Dml}[2]{\Dm{#1}^{\sst\le #2}}

\newcommand{\Dpg}[2]{\Dp{#1}^{\sst\ge #2}}

\newcommand{\Dkl}[2]{\Db{#1}_{\sst\sqsubseteq #2}}
\newcommand{\Dkg}[2]{\Db{#1}_{\sst\sqsupseteq #2}}

\newcommand{\Dpkg}[2]{\Dp{#1}_{\sst\sqsupseteq #2}}

\newcommand{\red}{{\mathrm{red}}}

\newcommand{\hto}{\hookrightarrow}
\newcommand{\ssm}{\smallsetminus}

\newcommand{\la}{\langle}
\newcommand{\ra}{\rangle}

\DeclareMathOperator{\im}{im}

\DeclareMathOperator{\skdeg}{sk\;deg}

\DeclareMathOperator{\step}{step}
\DeclareMathOperator{\Hom}{Hom}
\DeclareMathOperator{\Ext}{Ext}
\DeclareMathOperator{\End}{End}

\DeclareMathOperator{\Spec}{Spec}

\newtheorem{thm}{Theorem}[section]
\newtheorem{lem}[thm]{Lemma}
\newtheorem{prop}[thm]{Proposition}
\newtheorem{cor}[thm]{Corollary}

\theoremstyle{definition}
\newtheorem{defn}[thm]{Definition}

\theoremstyle{remark}
\newtheorem{rmk}[thm]{Remark}

\numberwithin{equation}{section}

\title[On the quasi-hereditary property for staggered sheaves]{On the
quasi-hereditary property for\\ staggered sheaves}
\author{Pramod N.~Achar}
\address{Department of Mathematics, Louisiana State University, Baton
Rouge, LA \ 70803}
\email{pramod@math.lsu.edu}
\thanks{The author was partially supported by NSF Grant DMS-0500873.}

\begin{document}

\begin{abstract}
Let $G$ be an algebraic group over an algebraically closed field, acting on a variety $X$ with finitely many orbits.  \emph{Staggered sheaves} are certain complexes of $G$-equivariant coherent sheaves on $X$ that seem to possess many remarkable properties.  In this paper, we construct ``standard'' and ``costandard'' objects in the category of staggered sheaves, and we prove that that category has enough projectives and injectives.
\end{abstract}

\maketitle

\section{Introduction}
\label{sect:intro}

Let $X$ be a variety over an algebraically closed field $\Bbbk$, and let $G$ be a linear algebraic group over $\Bbbk$ acting on $X$ with finitely many orbits.  {\it Staggered sheaves}~\cite{a, at:bs} are the objects in the heart of certain $t$-structure on the bounded derived category $\Db X$ of $G$-equivariant coherent sheaves on $X$.  The category of staggered sheaves, denoted $\cM(X)$, enjoys a growing list of remarkable properties, analogous in many ways to properties of $\ell$-adic mixed perverse sheaves~\cite{at:bs, at:pd}:
\begin{itemize}
\item Every object has finite length.  Simple objects arise via an ``$\cIC$'' functor and are parametrized by irreducible vector bundles on $G$-orbits.
\item There is a well-behaved notion of ``purity'' in $\Db X$, and every simple staggered sheaf is pure.
\item Every pure object in $\Db X$ is semisimple, {\it i.e.}, a direct sum of shifts of simple staggered sheaves.
\end{itemize}
In this paper, we add to this list as follows. First, we prove that
$\cM(X)$ is \emph{quasi-hereditary}, meaning that every simple object is a
quotient of some ``standard'' object and is contained in some
``costandard'' object.  (See Section~\ref{sect:abcat} for definitions.) 
This answers a question I was asked by David Vogan.  Second, we prove that
$\cM(X)$ has enough projectives and injectives.  These are analogues of results on perverse sheaves due to Mirollo--Vilonen~\cite{mv}.

The paper is organized as follows.  Section~\ref{sect:abcat} contains some generalities on quasi-hereditary abelian categories, and Section~\ref{sect:review} is a review of relevant facts about staggered sheaves.  In Section~\ref{sect:pushfd}, we prove that the functor of restriction to an open subscheme has both left and right adjoints in $\cM(X)$.  (In general, there are no such adjoints in $\Db X$, of course.)  We use those adjoints to construct standard and costandard objects in Section~\ref{sect:objects}.

The next two sections contain useful auxiliary results.  In Section~\ref{sect:full}, we show that the subcategory of staggered sheaves supported on some closed subscheme is Serre.  (A corollary is that staggered sheaves do not ``see'' nilpotent thickenings of schemes.)  Next, Section~\ref{sect:struc} gives an explicit description of the structure of standard and costandard objects.  We obtain several $\Ext^1$-vanishing results as corollaries.  The main theorem, on projectives and injectives in $\cM(X)$, is proved in Section~\ref{sect:extfin}.  Finally, Section~\ref{sect:exam} presents a brief example.

\subsection*{Acknowledgements}
I am grateful to Roman Bezrukavnikov and David Vogan for suggestions that led to this work, and to David Treumann for numerous fruitful and inspiring conversations.

\section{Preliminaries on Quasi-hereditary Categories}
\label{sect:abcat}

Let $\fA$ be a $\Bbbk$-linear abelian category.  Assume that $\fA$ has a
small skeleton; {\it i.e.}, that the class $S$ of isomorphism classes of
simple objects forms a set.  For each $s \in S$, fix a representative
$L(s)$.  We also assume that $\fA$ is of \emph{finite type}, meaning
that every object has finite length, and we write $[X:L(s)]$ for the multiplicity of $L(s)$ in any composition series for $X$.  We also assume
that
\[
\Hom(L(s),L(s)) \simeq \Bbbk \qquad\text{for all $s \in S$.}
\]
These assumptions imply that $\fA$ is also \emph{Hom-finite}: the
vector space $\Hom(X,Y)$ is finite-dimensional for any two objects $X, Y
\in \fA$.

Recall that if a simple object $L(s)$ admits a projective cover $(P,\phi)$
(where $\phi: P \to L(s)$ is a surjective morphism), it is unique up to
isomorphism, but in general not canonically so.  The same holds for
injective hulls $(I,\psi)$.  (See Lemma~\ref{lem:proj-iso} below, however.)

The next two elementary lemmas will be useful in the sequel.  Their proofs are routine and will be omitted.

\begin{lem}\label{lem:proj-mult}
Let $X$ be an object of $\fA$.  If $L(s)$ has a projective cover
$(P,\phi)$, then $[X : L(s)] = \dim \Hom(P,X)$.  If $L(s)$ has an
injective hull $(I,\psi)$, then $[X : L(s)] = \dim \Hom(X,I)$.\qed
\end{lem}

\begin{lem}\label{lem:proj-iso}
Assume that for some (and hence any) projective cover $(P,\phi)$ of
$L(s)$, we have $[P: L(s)] = 1$.  Then any two projective covers
of $L(s)$ are canonically isomorphic.\qed
\end{lem}

Assume henceforth that the set $S$ is equipped with a fixed order
$\preceq$. For $s \in
S$, let $\fA_{\preceq s}$ (resp.~$\fA_{\prec s}$) denote the Serre
subcategory generated by $\{L(t) \mid t \preceq s\}$ (resp.~$\{L(t) \mid t
\prec s\}$).

The following definitions are taken from~\cite{bez:qes}.  Related but
slightly different notions appear in~\cite{cps:fda} and~\cite{bgs}.

\begin{defn}
A \emph{standard cover} of $L(s)$ is a projective cover $(M(s),\phi(s))$
of $L(s)$ within the category $\fA_{\preceq s}$ with the property that
$[M(s) : L(s)] = 1$.  Similarly, a \emph{costandard hull} of
$L(s)$ is an injective hull $(N(s),\psi(s))$ within $\fA_{\preceq s}$ such
that $[N(s) : L(s)] = 1$.

$\fA$ is said to be \emph{quasi-hereditary} if every simple object has a
standard cover and a costandard hull.
\end{defn}

By Lemma~\ref{lem:proj-iso}, standard covers and costandard hulls are
unique (when they exist) up to canonical isomorphism.

\begin{lem}\label{lem:max-s}
The following conditions on an object $X \in \fA$ are equivalent:
\begin{enumerate}
\item $X \in \fA_{\preceq s}$.
\item $\Hom(M(t),X) = 0$ for all $t \succ s$.
\item $\Hom(X,N(t)) = 0$ for all $t \succ s$.
\end{enumerate}
\end{lem}
\begin{proof}
If $X \in \fA_{\preceq s}$, it is clear that conditions~(2) and~(3) hold. 
Now, suppose $X \notin \fA_{\preceq s}$.  We will prove by induction on
the length of $X$ that condition~(2) also fails; the proof that~(3) fails
is similar. Let $X'$ be a simple subobject of $X$, and let $X'' = X/X'$. 
Then either $X' \notin \fA_{\preceq s}$ or $X'' \notin \fA_{\preceq s}$. 
If $X' \notin \fA_{\preceq s}$, then $X' \simeq L(t)$ for some $t \succ
s$, so there is clearly a nonzero morphism $M(t) \to X'$.  Thus,
$\Hom(M(t),X) \ne 0$ as well.  On the other hand, if $X' \in \fA_{\preceq
s}$ but $X'' \notin \fA_{\preceq s}$, then, by assumption, there is some
$t \succ s$ such that $\Hom(M(t),X'') \ne 0$.  We also have $\Hom(M(t),X')
= \Ext^1(M(t),X') = 0$, so the natural morphism $\Hom(M(t),X) \to
\Hom(M(t),X'')$ is an isomorphism.  In particular, $\Hom(M(t),X) \ne 0$,
as desired.
\end{proof}

Assume $\fA$ is quasi-hereditary.  For any object $X \in \fA$ and any $s
\in S$, we define
\[
\la X:M(s) \ra = \dim \Hom(X,N(s))
\qquad\text{and}\qquad
\la X:N(s) \ra = \dim \Hom(M(s),X).
\]
Note that if $X$ is a projective cover $P(t)$ of a simple object $L(t)$,
Lemma~\ref{lem:proj-mult} gives us an alternate interpretation of $\dim
\Hom(X,N(s))$.  We see then that
\begin{equation}\label{eqn:recip}
 \la P(t):M(s) \ra = [ N(s):L(t) ].
\end{equation}
This is, of course, the famous ``Brauer--Humphreys reciprocity'' formula for
highest-weight categories~\cite{cps:fda}.  In such a category, the
projective cover of a simple object admits a \emph{standard filtration},
{\it i.e.}, a filtration whose subquotients are standard objects, and the
number $\la P(t):M(s) \ra$ is the precisely the multiplicity with which
$M(s)$ occurs in any standard filtration of $P(t)$.  (This follows from the
fact that $\Hom(M(t), N(s)) = 0$ for $s \ne t$.)

It is not true that projectives in an arbitrary quasi-hereditary
category necessarily admit standard filtrations, and the numbers $\la
P(t):M(s) \ra$ cannot always be interpreted as multiplicities. 
Nevertheless, a weak form of these ideas holds in great generality: the
proposition below tells us that for any object $X \in \fA$, the numbers
$\la X:M(s) \ra$ give information about the subquotients of a certain
canonical filtration of $X$.

Given a finite-dimensional $\Bbbk$-vector space $V$, consider the object
$V \otimes M(s)$.  Let us say that a quotient $Y$ of $V \otimes M(s)$ is
\emph{essential} if $[Y:L(s)] = \dim V$.  Equivalently, $Y$ is an
essential quotient if the kernel of the morphism $V \otimes M(s) \to Y$
contains no subobject isomorphic to $M(s)$.  Note that if $\Hom(X,N(s))
\ne 0$, it must be the case that $[X:L(s)] \ne 0$, so $\Hom(X,N(s))$ must
vanish for all but finitely many $s$.

\begin{prop}\label{prop:std-filt}
Assume $\fA$ is quasi-hereditary.  Given $X \in \fA$, let $s_1 \prec s_2
\prec \cdots \prec s_k$ be the elements of $S$ such that $\Hom(X,N(s)) \ne
0$.  There is a canonical decreasing filtration
\[
X = X_1 \supset X_2 \supset \cdots \supset X_k \supset X_{k+1} = 0
\]
such that $X_i/X_{i+1}$ is an essential quotient of $\Hom(X,N(s_i))^*
\otimes M(s_i)$.
\end{prop}
\begin{proof}
From Lemma~\ref{lem:max-s}, we see that $X \in \fA_{\preceq s_k}$, and
that $\Hom(M(t),X) = 0$ for $t \succ s_k$ but $\Hom(M(s_k),X) \ne 0$.
Consider the canonical morphism
\[
e: \Hom(M(s_k),X) \otimes M(s_k) \to X.
\]
Let $X_k$ denote the image of this morphism, and let $X' = X/X_k$.  Note
that $X_k$ is an essential quotient of $\Hom(M(s_k),X) \otimes M(s_k)$:
otherwise, there would be some nonzero $f \in \Hom(M(s_k),X)$ with $\Bbbk
f \otimes M(s_k) \subset \ker e$, but that is absurd:
\[
 e(\Bbbk f \otimes M(s_k)) = \im f.
\]
Next, consider the exact sequence
\[
0 \to \Hom(X_k,N(s)) \to \Hom(\Hom(M(s_k),X) \otimes M(s_k), N(s)) \to
\Hom(\ker e, N(s)).
\]
When $s \prec s_k$, the middle term vanishes, and therefore the first term
does as well.  When $s = s_k$, the last term vanishes, so the first two
are isomorphic to one another.  Note that since $M(s_k)$ is a projective
object in $\fA_{\preceq s_k}$, and
$N(s_k)$ is injective, there is a nondegenerate pairing
\[
 \Hom(M(s_k),X) \otimes \Hom(X,N(s_k)) \to \Hom(M(s_k),N(s_k)) \simeq
\Bbbk.
\]
Thus, we have a sequence of isomorphisms
\begin{multline*}
\Hom(\Hom(M(s_k),X), \Hom(M(s_k),N(s_k))) \simeq \Hom(\Hom(M(s_k),X),
\Bbbk) \\
\simeq \Hom(M(s_k),X)^* \simeq \Hom(X,N(s_k)).
\end{multline*}
Next, consider the sequence
\[
0 \to \Hom(X',N(s)) \to \Hom(X,N(s)) \to \Hom(X_k,N(s)).
\]
Because $\fA_{\preceq s_k}$ is a Serre subcategory, both $X_k$ and $X'$
belong to it.  Thus, all three terms above vanish when $s \succ s_k$.
We saw above that the last term vanishes when $s \prec s_k$, and that the
map $\Hom(X,N(s)) \to \Hom(X_k,N(s))$ is an isomorphism when $s = s_k$. 
Combining these observations, we find that
\[
 \Hom(X',N(s)) \simeq
\begin{cases}
 \Hom(X,N(s)) & \text{if $s \ne s_k$,} \\
0 & \text{if $s = s_k$.}
\end{cases}
\]
We have shown that $X_k$ is an essential quotient of $\Hom(X,N(s_k))^*
\otimes M(s_k)$, so the result follows by induction.
\end{proof}

\section{Review of Staggered Sheaves}
\label{sect:review}

In this section, we fix notation and briefly review relevant facts about
staggered sheaves.  Let $X$ denote a (not necessarily reduced) scheme of
finite type over $\Bbbk$, and let $G$ denote a linear algebraic group over
$\Bbbk$, acting on $X$ with finitely many orbits.  Here, and throughout
the paper, an \emph{orbit} will mean a reduced, locally closed
$G$-invariant subscheme containing no smaller nonempty $G$-invariant
subscheme.  Let $\cg X$ denote the category of $G$-equivariant coherent
sheaves on $X$, and let $\Db X$ denote its bounded derived category.  We assume throughout that $\cg X$ has enough locally free objects.

Let $\orb X$ denote the set of $G$-orbits on $X$, and let
$\orbvb X$ denote the set of isomorphism classes of pairs $(C,\cL)$, where
$C \in \orb X$ and $\cL \in \cg C$ is an irreducible $G$-equivariant vector
bundle on $C$.  The category of staggered sheaves $\cM(X)$ depends on two
choices: a \emph{perversity}, which is simply a function $r: \orb X \to
\Z$, and an \emph{$s$-structure}, which is a certain kind of increasing
filtration of $\cg X$ (see~\cite{a}). We will not review the rather lengthy
and complicated definition of an $s$-structure here. Instead, we recall
only that an $s$-structure allows us to assign to each pair $(C,\cL) \in
\orbvb X$ a certain integer, denoted $\step \cL$.

We regard both the perversity and the $s$-structure as fixed, once and for
all.  Moreover, we assume that the $s$-structure is ``recessed'' and
``split,'' so that the results of~\cite{at:bs, at:pd} are available.  In particular, the results of~\cite[Section~8]{at:bs} allow us to define $\cM(X)$ with no assumption on $r$.  (In contrast, the original construction in~\cite{a} required $r$ to obey stringent inequalities.)  For examples of $s$-structures, see~\cite{as:flag, t}.

Staggered sheaves on a single orbit $C \subset X$ are easy to describe. 
Given an irreducible vector bundle $\cL \in \cg C$, let $d_\cL = \step \cL
- r(C)$.  The \emph{staggered $t$-structure} on $\Db C$, denoted $(\ru\Dl
C0,\ru\Dg C0)$, is the unique $t$-structure on $\Db C$ whose heart contains
all objects of the form $\cL[d_\cL]$.  By definition, $\cM(C) = \ru\Dl C0
\cap \ru\Dg C0$. 

Next, let $\Dm C$ denote the bounded-above derived category of $\cg C$, and
let $\ru\Dml C0$ denote the full subcategory consisting of objects $\cF$
such that $h^i(\cF)[-i] \in \ru\Dl C0$ for all $i$.  Let $i_C: \overline C
\hto X$ denote the inclusion of the closure of $C$.  The \emph{staggered
$t$-structure} $(\ru\Dl X0, \ru\Dg X0)$ on $\Db X$ is characterized by the
property that $\cF \in \ru\Dl X0$ if and only if $Li_C^*\cF|_C \in \ru\Dml
C0$ for all orbits $C$.  (This condition must be stated in terms of
$\ru\Dml C0$ because $Li_C^*$ does not, in general, take values in the
bounded derived category.) 

A dual version of the description above is as follows.  Let $\qg X$ denote
the category of $G$-equivariant quasicoherent sheaves, and let $\Dp X$
denote the full subcategory of the bounded-below derived category of $\qg
X$ consisting of objects with coherent cohomology.  The full subcategory
$\ru\Dpg X0 \subset \Dp X$ consists of objects $\cG$ such that
$\Hom(\cF,\cG) = 0$ for all $\cF \in \ru\Dl X{-1}$.  It turns out that $\cG
\in \ru\Dpg X0$ if and only if $Ri_C^!\cG|_C \in \ru\Dpg C0$ for all orbits
$C$. 

Let $j: U \hto X$ be the inclusion of a $G$-invariant open subscheme, and
let $i: Z \hto X$ be the inclusion of $G$-invariant closed subscheme.  The
restriction functor $j^*: \Db X \to \Db U$ and the push-forward functor
$i_*: \Db Z \to \Db X$ are $t$-exact for the staggered $t$-structure. 

We now turn to the construction of simple objects in $\cM(X)$.  Let $j: U
\hto X$ be as above.  Define two new functions $\ufl r, \ush r: \orb X \to
\Z$ by 
\begin{equation}\label{eqn:sharpflat}
\ufl r(C) =
\begin{cases}
r(C) & \text{if $C \subset U$,} \\
r(C) - 1 & \text{if $C \not\subset U$,}
\end{cases}
\qquad
\ush r(x) =
\begin{cases}
r(C) & \text{if $C \subset U$,} \\
r(C) + 1 & \text{if $C \not\subset U$.}
\end{cases}
\end{equation}
By~\cite[Proposition~8.7]{at:bs}, for any object $\cF \in \cM(U)$, there is
(up to isomorphism) a unique object in $\Db X$, denoted $\ru j_{!*}\cF$,
such that 
\begin{equation}\label{eqn:ic-defn}
\ru j_{!*}\cF|_U \simeq \cF
\qquad\text{and}\qquad
\ru j_{!*}\cF \in \flatru \Dl X0 \cap \sharpru \Dg X0.
\end{equation}
The assignment $\cF \mapsto \ru j_{!*}\cF$ defines a faithfully full
functor $\ru j_{!*}: \cM(U) \to \cM(X)$, and its essential image, denoted
$\cMn(X)$, is a Serre subcategory of $\cM(X)$.  In particular, if $\cF \in
\cM(U)$ is a simple object, then $\ru j_{!*}\cF$ is a simple object of
$\cM(X)$. 

More generally, given an orbit $C \subset X$, let $\partial C$ denote the
closed set $\overline C \ssm C$ (not regarded as having a fixed scheme
structure), and let $U_C = X \ssm \partial C$.  This is an open subscheme
of $X$ containing $C$ as a closed orbit.  Form the diagram of inclusions 
\[
\xymatrix{
C \ar[d]_{j_C} \ar[r]^{t_C} & U_C \ar[d]^{h_C} \\
\overline C \ar[r]_{i_C} & X}
\]
If $\cL \in \cg C$ is an irreducible vector bundle, $\cL[d_\cL]$ is a
simple object of $\cM(C)$.  We associate to $(C,\cL)$ a simple object of
$\cM(X)$, denoted $\cIC(C,\cL)$, by 
\[
\cIC(C,\cL) = i_{C*}(\ru j_{C!*}\cL[d_\cL]) \simeq
h_{C!*}(t_{C*}\cL[d_\cL]). 
\]
(The isomorphism $i_{C*} \circ \ru j_{C!*} \simeq h_{C!*} \circ t_{C*}$
follows from the $t$-exactness of $i_{C*}$ and the uniqueness
property~\eqref{eqn:ic-defn} for $h_{C!*}$.)  All simple objects of
$\cM(X)$ arise in this way.  Thus, the set of isomorphism classes of simple
objects is in bijection with $\orbvb X$. 

The results of~\cite[Section~9]{at:pd} associate to the perversity $r$ a
collection of thick subcategories $(\{\Dkl Xw\}, \{\Dkg Xw\})_{w \in \Z}$
that behave much like the weight filtration on $\ell$-adic mixed sheaves. 
These
subcategories, together called the \emph{skew co-$t$-structure} in {\it
loc.~cit.}, enjoy the following properties:
\begin{align*}
 \Dkl X{w-1} &\subset \Dkl Xw &
\Dkl X{w-1}[1] &= \Dkl Xw \\
 \Dkg X{w-1} &\supset \Dkg Xw &
\Dkg X{w-1}[1] &= \Dkg Xw
\end{align*}
An object of $\Dkl Xw \cap \Dkg Xw$ is said to be \emph{skew-pure} of skew
degree $w$.  The Skew Purity Theorem~\cite[Theorem~10.2]{at:pd} states that
every staggered
sheaf has a canonical filtration with skew-pure subquotients, and in
particular that a simple staggered sheaf is skew-pure.  Indeed, the skew
degree of such an object is given by
\[
 \skdeg \cIC(C,\cL) = 2d_\cL - \dim C. 
\]
(Skew degrees in~\cite{at:pd} differ from this formula by the addition of some constant depending on the choice of a dualizing complex, but we may ignore that constant here.)  The Skew Decomposition Theorem~\cite[Theorem~11.5]{at:pd} states that every skew-pure object is semisimple.

We conclude with the following useful fact.  Here, and throughout the paper, we write $\Hom^n(\cF,\cG)$ for $\Hom(\cF, \cG[n])$ for any two objects $\cF, \cG \in \Db X$, or even $\cF \in \Dm X$ and $\cG \in \Dp X$.  See~\cite[Proposition~2]{bez:pc} for the proof.

\begin{lem}\label{lem:hom-les}
Let $j: U \hto X$ be the inclusion of a $G$-invariant open subscheme, and let $\uZ = X \ssm U$ denote the complementary closed subset.  For any two objects $\cF_1 \in \Dm X$ and $\cF_2 \in \Dp X$, there is a long exact sequence
\begin{multline}\label{eqn:hom-les}
\cdots \to \Hom^{-1}(j^*\cF_1|_U,j^*\cF_2|_U) \to
\lim_{\substack{\to \\ Z'}} \Hom(Li^*_{Z'}\cF_1, Ri^!_{Z'}\cF_2) \to \\
\Hom(\cF_1,\cF_2) \to 
\Hom(j^*\cF_1|_U,\cF_2|_U) \to 
\lim_{\substack{\to \\ Z'}} \Hom^1(Li^*_{Z'}\cF, Ri^!_{Z'}\cG) \to
 \cdots,
\end{multline}
where $i_{Z'}: Z' \hto X$ ranges over all closed subscheme structures on $\uZ$.
\qed
\end{lem}

\section{Restriction to an Open Subscheme}
\label{sect:pushfd}

Let $j: U \hto X$ be the inclusion of a $G$-invariant open subscheme.  In this section, we construct left and right adjoints to the restriction functor $j^*: \cM(X) \to \cM(U)$.

The perversities $\ufl r$ and $\ush r$ defined in~\eqref{eqn:sharpflat} give rise to their own staggered $t$-structures on $\Db X$, and hence their own intermediate extension functors $\flatru j_{!*}$ and $\sharpru j_{!*}$.
Now, $\flatru j_{!*}$ takes values in $\ffru\Dl X0 \cap \sfru\Dg X0$.  But $\ush(\ufl r) = r$, and clearly $\ffru\Dl X0 \subset \ru\Dl X0$, so we see that $\flatru j_{!*}$ actually takes values in $\cM(X)$.  The same holds for $\sharpru j_{!*}$, by similar reasoning.

We introduce the notation
\[
\ru j_! = \flatru j_{!*}
\qquad\text{and}\qquad
\ru j_* = \sharpru j_{!*}
\]
for these functors regarded as functors $\cM(U) \to \cM(X)$.

\begin{prop}\label{prop:adjt}
Let $j: U \hto X$ be the inclusion of a $G$-invariant open subscheme, and let $\uZ$ denote the closed subset complementary to $U$.  For $\cF \in \cM(U)$ and $\cG \in \cM(X)$, there are canonical isomorphisms
\[
\Hom(\ru j_!\cF, \cG) \simeq \Hom(\cF, j^*\cG) 
\qquad\text{and}\qquad
\Hom(\cG, \ru j_*\cF) \simeq \Hom(j^*\cG, \cF).
\]
There is a canonical surjective morphism $\ru j_!\cF \to \ru j_{!*}\cF$ whose kernel is supported on $\uZ$, and a canonical injective morphism $\ru j_{!*}\cF \to \ru j_*\cF$ whose cokernel is supported on $\uZ$.
\end{prop}
\begin{proof}
Let us apply Lemma~\ref{lem:hom-les} with $\cF_1 = \ru j_!\cF$ and $\cF_2 = \cG$.  Since $Li_{Z'}^*(\ru j_!\cF) \in \ffru\Dml {Z'}0 = \ru\Dml {Z'}{-2}$ and $Ri_{Z'}^!\cG \in \ru\Dpg X0$ for any closed subscheme structure $i_{Z'}: Z' \hto X$ on the complement of $U$, we see that
\[
\Hom(Li_{Z'}^*(\ru j_!\cF), Ri_{Z'}^!\cG) =
\Hom^1(Li_{Z'}^*(\ru j_!\cF), Ri_{Z'}^!\cG) = 0.
\]
It follows from~\eqref{eqn:hom-les} that $\Hom(\ru j_!\cF, \cG) \simeq \Hom(\cF, j^*\cG)$.  The proof of the statement that $\Hom(\cG, \ru j_*\cF) \simeq \Hom(j^*\cG, \cF)$ is similar.

The first adjointness statement gives us an isomorphism
$\Hom(\ru j_!\cF, \ru j_{!*}\cF) \simeq \Hom(\cF,\cF)$, and hence a canonical morphism $\ru j_!\cF \to \ru j_{!*}\cF$.  Since the restriction of this map to $U$ is an isomorphism, its kernel and cokernel must both be supported on $\uZ$.  But $\ru j_{!*}\cF$, an object of the Serre subcategory $\cMn(X)$, has no nonzero quotient supported on $\uZ$, so the morphism $\ru j_!\cF \to \ru j_{!*}\cF$ must be surjective.  Similarly, the second adjointness statement gives us a canonical injective morphism $\ru j_{!*}\cF \to \ru j_*\cF$ with cokernel supported on $\uZ$.
\end{proof}

Let us make note of a particular instance of the preceding proposition.  For an orbit $C \subset X$ and an irreducible vector bundle $\cL \in \cg C$, we put
\begin{align*}
M(C,\cL) &= i_{C*} \ru j_{C!}\cL[d_\cL] \simeq \ru h_{C!}t_{C*}\cL[d_\cL],
\\
N(C,\cL) &= i_{C*} \ru j_{C*}\cL[d_\cL] \simeq \ru h_{C*}t_{C*}\cL[d_\cL].
\end{align*}

\begin{prop}\label{prop:ker-supp}
For any $(C,\cL) \in \orbvb X$, there are canonical nonzero morphisms
\[
\phi: M(C,\cL) \to \cIC(C,\cL)
\qquad\text{and}\qquad
\psi: \cIC(C,\cL) \to N(C,\cL).
\]
The kernel of $\phi$ and cokernel of $\psi$ are both supported on $\partial C$. \qed
\end{prop}

\section{Standard and Costandard Objects}
\label{sect:objects}

In this section, we prove that $\cM(X)$ is quasi-hereditary.
We begin with a result about closed orbits.

\begin{prop}\label{prop:closed-semis}
Let $i_C: C \hto X$ be the inclusion of a closed orbit.  Then $\cM(C)$ is a
semisimple category, and $i_{C*}: \cM(C) \to \cM(X)$ is an embedding of it
as a Serre subcategory of $\cM(X)$. 
\end{prop}
\begin{proof}
Note that since $C$ is closed, we have $\cIC(C,\cL) \simeq
i_{C*}\cL[d_\cL]$ for any irreducible vector bundle $\cL \in \cg C$.  Let
$\cM'(C) \subset \cM(X)$ be the Serre subcategory of $\cM(X)$ generated by
objects of the form $\cIC(C,\cL)$.  This is, of course, the smallest Serre
subcategory of $\cM(X)$ containing $i_{C*}(\cM(C))$.  We will show that
$\cM'(C)$ is semisimple.  Since $i_{C*}$ is faithful, that implies that
$\cM(C)$ is semisimple.  Moreover, since $i_{C*}(\cM(C))$ is closed under
direct sums and contains all simple objects of $\cM'(C)$, it also implies
that $i_{C*}(\cM(C)) = \cM'(C)$.

To show that $\cM'(C)$ is semisimple, it suffices to show that
\begin{equation}\label{eqn:ext-orbit}
\Ext^1(i_{C*}\cL[d_\cL], i_{C*}\cL'[d_{\cL'}]) = 0
\end{equation}
for any two irreducible vector bundles $\cL, \cL' \in \cg C$.  If $d_\cL >
d_{\cL'} + 1$, this vanishing is obvious.  On the other hand, if $d_\cL \le
d_{\cL'}$, then we have $\skdeg i_{C*}\cL[d_\cL] \le \skdeg
i_{C*}\cL'[d_{\cL'}]$.  In this case,~\eqref{eqn:ext-orbit} follows
from~\cite[Proposition~11.2]{at:pd}.
\end{proof}

Suppose we apply the preceding proposition to the closed embedding $t_C: C \hto U_C$.  Since $h_{C!*}$ embeds $\cM(U_C)$ as a Serre subcategory of $\cM(X)$, we obtain the following result.

\begin{cor}\label{cor:ic-semis}
For any orbit $C \subset X$, the functor $\cIC(C,\cdot): \cM(C) \to \cM(X)$ is an embedding of $\cM(C)$ as a semisimple Serre subcategory of $\cM(X)$.
\end{cor}

\begin{prop}\label{prop:hom-ext}
Let $C, C' \subset X$ be orbits, and let $\cL \in \cg C$ and $\cL' \in \cg
{C'}$ be irreducible vector bundles.  Assume that either (i)~$C \not\subset
\overline C'$, or (ii)~$C = C'$ and $\cL \not\simeq \cL'$.
Then we have
\begin{align*}
\Hom(M(C,\cL), \cIC(C', \cL')) &= \Ext^1(M(C,\cL), \cIC(C',\cL')) = 0 \\
\Hom(\cIC(C',\cL'), N(C,\cL)) &= \Ext^1(\cIC(C',\cL'), N(C,\cL)) = 0.
\end{align*}
\end{prop}
\begin{proof}
Form the exact sequence~\eqref{eqn:hom-les} with $\cF_1 = M(C,\cL)$, $\cF_2 = \cIC(C',\cL')$, and $U = U_C$.  For any closed subscheme structure $i_{Z'}: Z' \hto X$ on $\partial C$, we have $Li_{Z'}^*M(C,\cL) \in \q\Dml {Z'}{-2}$, and $Ri_{Z'}^!\cIC(C',\cL') \in \q\Dpg {Z'}0$.  It follows that
\[
\Hom(Li_{Z'}^*M(C,\cL), Ri_{Z'}^!\cIC(C',\cL')) =
\Hom^1(Li_{Z'}^*M(C,\cL), Ri_{Z'}^!\cIC(C',\cL')) = 0.
\]
Note that $h_C^*M(C,\cL)$ is isomorphic to $t_{C*}\cL[d_\cL]$; in
particular, this is a simple object of $\cM(U_C)$.  Next,
$h_C^*\cIC(C',\cL')$ is either $0$ or a simple object, and in the latter
case, it is distinct from $t_{C*}\cL[d_\cL]$.  Therefore,
$\Hom(h_C^*M(C,\cL), h_C^*\cIC(C',\cL')) = 0$.  Moreover, under the
assumptions in the statement of the proposition, the support of
$h_C^*\cIC(C',\cL')$ is either disjoint from $C$ or equal to $C$.  In the
first case, it is obvious that $\Ext^1(h_C^*M(C,\cL), h_C^*\cIC(C',\cL')) =
0$, and in the second, this follows from
Proposition~\ref{prop:closed-semis}.  We can then see from the exact
sequence~\eqref{eqn:hom-les} that $\Hom(M(C,\cL), \cIC(C', \cL')) =
\Ext^1(M(C,\cL), \cIC(C',\cL')) = 0$.  The proofs of the statements with
$N(C,\cL)$ are similar and will be omitted.
\end{proof}

Together, Propositions~\ref{prop:ker-supp} and~\ref{prop:hom-ext} give us
the following result.

\begin{thm}\label{thm:main}
Let $\prec$ be any total order on $\orbvb X$ such that $C \subsetneq
\overline C'$ implies $(C,\cL) \prec (C',\cL')$ for all $\cL \in \cg C$ and
$\cL' \in \cg {C'}$.  With respect to this order, $M(C,\cL)$ is a standard
cover of $\cIC(C,\cL)$, and $N(C,\cL)$ is a costandard hull.  Thus,
$\cM(X)$ is a quasi-hereditary category.\qed
\end{thm}

\section{Staggered Sheaves on Closed Subschemes}
\label{sect:full}

We will now make use of standard and costandard objects to show that for any $G$-invariant closed subscheme $i: Z \hto X$, $\cM(Z)$ embeds as a Serre subcategory of $\cM(X)$.  We begin with a result on $\Hom$- and $\Ext^1$-groups.

\begin{prop}\label{prop:full}
Let $i: Z \hto X$ be a $G$-invariant closed subscheme.  For any $\cF, \cG \in \cM(Z)$, we have
\begin{align}
\Hom_Z(\cF,\cG) &\simeq \Hom_X(i_*\cF,i_*\cG)\label{eqn:hom-full} \\
\Ext^1_Z(\cF,\cG) &\simeq \Ext^1_X(i_*\cF,i_*\cG)\label{eqn:ext-full}
\end{align}
\end{prop}
(To avoid confusion, in this proposition and in its proof, we explicitly label all $\Hom$- and $\Ext$-groups with the name of the scheme over which that group is to be computed.)
\begin{proof}
Assume that $\cF$ and $\cG$ are both nonzero (otherwise, the statement is trivial).  We proceed by noetherian induction, and assume the statement is already known when either $Z$ or $X$ is replaced by some proper closed subscheme.  

We begin by proving the proposition in the case where $\cF$ and $\cG$, and
hence $i_*\cF$ and $i_*\cG$, are both simple.  Suppose $i_*\cF \simeq
\cIC(C,\cL)$ and $i_*\cG \simeq \cIC(C',\cL')$. In the case where $C = C'$,
Corollary~\ref{cor:ic-semis} tells us that both sides
of~\eqref{eqn:hom-full} are isomorphic to
$\Hom_C(\cL[d_\cL],\cL'[d_{\cL'}])$, and that both sides
of~\eqref{eqn:ext-full} vanish.    Henceforth, assume $C \ne C'$.  Both
sides of~\eqref{eqn:hom-full} automatically vanish, since there are no
nonzero morphisms between nonisomorphic simple objects.  To
prove~\eqref{eqn:ext-full}, we must consider the various ways in which $C$
and $C'$ may be related.

Suppose first that $C \not\subset \overline C{}'$ and $C' \not\subset
\overline C$.  Let $U = X \ssm (\overline C \cap \overline C{}')$.  Then
$\cIC(C,\cL)|_U$ and $\cIC(C',\cL')|_U$ have disjoint supports, so
$\Ext^1_U(\cIC(C,\cL)|_U, \cIC(C',\cL')|_U) = 0$.  Also, for any closed
subscheme structure $i_{Z'}: Z' \hto X$ on $\overline C \cap \overline
C{}'$, we have $Li_{Z'}^*\cIC(C,\cL) \in \q\Dml {Z'}{-1}$ and
$Ri_{Z'}^!\cIC(C',\cL') \in \q\Dpg {Z'}{1}$, so it follows that
$\Hom^1_{Z'}(Li_{Z'}^*\cIC(C,\cL), Ri_{Z'}^!\cIC(C',\cL')) = 0$, and then
\begin{equation}\label{eqn:ext-incomp}
\Ext^1_X(\cIC(C,\cL), \cIC(C',\cL')) = 0
\end{equation}
by Lemma~\ref{lem:hom-les}.  The
same reasoning applies to $Z$, so both sides of~\eqref{eqn:ext-full} above
vanish.

Next, suppose that $C \subset \overline C{}'$.  Let $J =
N(C',\cL')/\cIC(C',\cL')$, and consider the exact sequence
\begin{multline*}
\Hom_X(\cIC(C,\cL), N(C',\cL')) \to \Hom_X(\cIC(C,\cL), J) \to \\
\Ext^1_X(\cIC(C,\cL), \cIC(C',\cL')) \to \Ext^1_X(\cIC(C,\cL),
N(C',\cL')). 
\end{multline*}
The first and last terms vanish by Proposition~\ref{prop:hom-ext}, so the
middle two terms are isomorphic.  By Proposition~\ref{prop:ker-supp}, there
is some closed subscheme structure $\kappa: Y \hto X$ on $\partial C'$ on
which $J$ is supported: $J \simeq \kappa_*J'$ for some $J' \in \cM(Y)$. 
Since $C \subsetneq \overline C{}'$, we likewise have $\cIC(C,\cL) \simeq
\kappa_*\cF'$ for some simple object $\cF' \in \cM(Y)$.  Since $Y$ is
strictly smaller than $Z$, we know by the inductive assumption that
$\Hom_X(\cIC(C,\cL),J) \simeq \Hom_Y(\cF',J')$.  The same reasoning also
applies to $Z$, so we have 
\[
\Ext^1_Z(\cIC(C,\cL), \cIC(C',\cL')) \simeq
\Ext^1_X(\cIC(C,\cL), \cIC(C',\cL')) \simeq
\Hom_Y(\cF', J'),
\]
as desired.  A similar argument using $M(C,\cL)$ applies to the case where
$C' \subset \overline C$.  This completes the proof of the proposition in
the case where $\cF$ and $\cG$ are both simple. 

For the general case, we proceed by induction on the sum of the lengths of
$\cF$ and $\cG$.  Suppose that $\cF$ is not simple, and find some short
exact sequence 
\[
0 \to \cF' \to \cF \to \cF'' \to 0
\]
with $\cF'$ and $\cF''$ both nonzero (and therefore of strictly smaller
length than $\cF$).  Consider the eight-term exact sequence
\begin{multline*}
0 \to \Hom(\cF'',\cG) \to \Hom(\cF,\cG) \to \Hom(\cF',\cG) \to \\
\Ext^1(\cF'',\cG) \to \Ext^1(\cF,\cG) \to \Ext^1(\cF',\cG) \to
\Hom^2(\cF'',\cG),
\end{multline*}
together with its analogue obtained by applying $i_*$ to every object.  By assumption, the four morphisms
\begin{align*}
\Hom(\cF',\cG) &\to \Hom(i_*\cF',i_*\cG), &
\Ext^1(\cF',\cG) &\to \Ext^1(i_*\cF',i_*\cG), \\
\Hom(\cF'',\cG) &\to \Hom(i_*\cF'',i_*\cG), &
\Ext^1(\cF'',\cG) &\to \Ext^1(i_*\cF'',i_*\cG)
\end{align*}
are isomorphisms, and $\Hom^2(\cF'',\cG) \to \Hom^2(i_*\cF'',i_*\cG)$
is injective because $i_*$ is faithful.  By two applications of the five
lemma, we see that $\Hom(\cF,\cG) \to \Hom(i_*\cF,i_*\cG)$ and
$\Ext^1(\cF,\cG) \to \Ext^1(i_*\cF,i_*\cG)$ are isomorphisms, as desired. 
A similar argument establishes~\eqref{eqn:hom-full}
and~\eqref{eqn:ext-full} in the case where $\cF$ is simple but $\cG$ is
not.
\end{proof}

\begin{thm}\label{thm:serre}
Let $i: Z \hto X$ be a $G$-invariant closed subscheme.  Then $i_*(\cM(Z))$ is a Serre subcategory of $\cM(X)$.
\end{thm}
\begin{proof}
The isomorphism~\eqref{eqn:hom-full} tells us that the functor $i_*: \cM(Z) \to \cM(X)$ is full.  It remains to show that the essential image of $i_*$ is stable under extensions.  Consider short exact sequence
\[
0 \to i_*\cF' \to \cF \to i_*\cF'' \to 0
\]
in $\cM(X)$, with $\cF', \cF'' \in \cM(Z)$.  The object $\cF$ is in the essential image of $i_*$ if and only if the corresponding element of $\Ext^1(i_*\cF'',i_*\cF')$ is in the image of the natural map $\Ext^1(\cF'',\cF') \to \Ext^1(i_*\cF'',i_*\cF')$, and the latter is surjective by~\eqref{eqn:ext-full}.
\end{proof}

In particular, since all simple staggered sheaves on a nonreduced scheme are supported on the associated reduced scheme, we obtain the following result.

\begin{cor}
Suppose $X$ is not reduced, and let $t: X_\red \to X$ be the inclusion of the associated reduced scheme.  Then $t_* : \cM(X_\red) \to \cM(X)$ is an equivalence of categories. \qed
\end{cor}

\section{Structure of Standard and Costandard Objects}
\label{sect:struc}

Let $K(C,\cL)$ denote the kernel of the canonical morphism $M(C,\cL) \to \cIC(C,\cL)$, and let $J(C,\cL)$ denote the cokernel of the canonical morphism $\cIC(C,\cL) \to N(C,\cL)$.  The goal of this section is to describe the structure of $K(C,\cL)$ and $J(C,\cL)$.

\begin{lem}\label{lem:ext-std}
Let $(C,\cL), (C',\cL') \in \orbvb X$, and assume that $C \not\subset \partial C'$.  There are natural isomorphisms
\begin{align*}
\Hom(K(C,\cL), \cIC(C',\cL')) &\simeq \Ext^1(\cIC(C,\cL), \cIC(C',\cL')), \\
\Hom(\cIC(C',\cL'), J(C,\cL)) &\simeq \Ext^1(\cIC(C',\cL'), \cIC(C,\cL)).
\end{align*}
\end{lem}
\begin{proof}
If $(C,\cL) \simeq (C',\cL')$, then both $\Hom$-groups vanish by
Proposition~\ref{prop:ker-supp}, and both $\Ext^1$-groups vanish by
Corollary~\ref{cor:ic-semis}.  Suppose now that $(C,\cL) \not\simeq
(C',\cL')$.  The first isomorphism above then comes from the long exact
sequence associated to $0 \to K(C,\cL) \to M(C,\cL) \to \cIC(C,\cL) \to 0$
using Proposition~\ref{prop:hom-ext}.  The second isomorphism is similar,
using $N(C,\cL)$ instead.
\end{proof}

\begin{lem}\label{lem:ext-finite-below}
For any $(C,\cL) \in \orbvb X$, there are only finitely many pairs
$(C',\cL') \in \orbvb X$ with $C \not\subset \partial C'$ such that either 
\[
\Ext^1(\cIC(C,\cL), \cIC(C',\cL')) \ne 0
\qquad\text{or}\qquad
\Ext^1(\cIC(C',\cL'), \cIC(C,\cL)) \ne 0.
\]
\end{lem}
\begin{proof}
Clearly, $\Hom(K(C,\cL), \cIC(C',\cL'))$ and $\Hom(\cIC(C',\cL'),
J(C,\cL))$ vanish for all but finitely many $(C',\cL')$, so this statement
follows from the previous lemma.
\end{proof}

\begin{thm}\label{thm:std-struc}
Suppose $\skdeg \cIC(C,\cL) = w$.  Then $K(C,\cL)$ is skew-pure of degree $w-1$, and there is a natural isomorphism
\[
K(C,\cL) = \bigoplus_{\{(C',\cL') \in \orbvb X \mid C' \subset \partial C\}}
\Ext^1(\cIC(C,\cL), \cIC(C',\cL'))^* \otimes \cIC(C',\cL').
\]
Similarly, $J(C,\cL)$ is skew-pure of degree $w+1$, and there is a natural isomorphism
\[
J(C,\cL) = \bigoplus_{\{(C',\cL') \in \orbvb X \mid C' \subset \partial C\}}
\Ext^1(\cIC(C',\cL'), \cIC(C',\cL')) \otimes \cIC(C',\cL').
\]
\end{thm}
\begin{proof}
By the Skew Purity Theorem~\cite[Theorem~10.2]{at:pd}, $M(C,\cL)$ admits a
canonical filtration
\[
0 = M_{\sqsubseteq w -k} \subset \cdots \subset M_{\sqsubset w - 1} \subset M_{\sqsubseteq w} = M(C,\cL)
\]
where each $M_{\sqsubseteq w - i} \in \Dkl X{w-i}$, and $M_{\sqsubseteq w -
i}/M_{\sqsubseteq w - i - 1}$ is skew-pure of degree $w - i$.  (In our
case, we know that $M_{\sqsubseteq w}$ must be the last step of the
filtration, because the unique simple quotient of $M(C,\cL)$ is skew-pure
of degree $w$.)  Let us begin by showing that in fact $M_{\sqsubset w-2} =
0$.  It suffices to show that $\Hom(\cF, M(C,\cL)) = 0$ for all $\cF \in
\Dkl X{w-2}$.  We employ the exact sequence of Lemma~\ref{lem:hom-les} with
$U = U_C$ and $\uZ = \partial C$.  Clearly, $M(C,\cL)|_{U_C} \simeq
\cIC(C,\cL)|_{U_C} \in \Dkg Uw$, so $\Hom(\cF|_{U_C}, M(C,\cL)|_{U_C}) =
0$.

The Skew Purity Theorem also tells us that any simple staggered sheaf is
skew-pure.  Since $M(C,\cL)$ arises by the functor $\ru j_{C!} = \flatru
j_{C!*}$, we see that $M(C,\cL)$ is actually a simple object in the
category $\flatru \cM(X)$ of staggered sheaves defined with respect to the
perversity $\ufl r$, and it is skew-pure of degree $w$ with respect to the
notion of purity defined with respect to $\ufl r$.  In the notation
of~\cite{at:pd}, we have $M(C,\cL) \in
{}_{\sst \llcorner \flat r \lrcorner}\Dkl Xw \cap
{}_{\sst \llcorner \flat r \lrcorner}\Dkg Xw$.
It follows that for any subscheme structure $i_{Z'}: Z' \hto X$ on $\partial C$, we have
$Ri_{Z'}^!M(C,\cL) \in 
{}_{\sst \llcorner \flat r \lrcorner}\Dpkg {Z'}w$.
Since $\ufl r(C') = r(C') - 1$ for all $C' \subset \partial C$, this
assertion is equivalent to $Ri_{Z'}^!M(C,\cL) \in \Dpkg {Z'}{w-1}$, where
the last category is defined as usual with respect to the perversity
$r$. It follows that $\Hom(Li_{Z'}^*\cF, Ri_{Z'}^!M(C,\cL)) = 0$, so
$\Hom(\cF, M(C,\cL)) = 0$, and $M_{\sqsubseteq w-2} = 0$, as desired.

Next, according to the Skew Decomposition
Theorem~\cite[Theorem~11.5]{at:pd}, any skew-pure object is semisimple.  In
particular, $M_{\sqsubseteq w - 1}$ and $M(C,\cL)/M_{\sqsubseteq w-1}$ are
both semisimple.  Because $M(C,\cL)$ has a unique simple quotient, we must
in fact have $M(C,\cL)/M_{\sqsubseteq w-1} \simeq \cIC(C,\cL)$, and then we
identify $M_{\sqsubseteq w-1}$ with $K(C,\cL)$.  We may write
\[
K(C,\cL) \simeq \bigoplus_{\{(C',\cL') \in \orbvb X \mid C' \subset \partial C\}} V_{C',\cL'} \otimes \cIC(C',\cL')
\]
for some vector spaces $V_{C',\cL'}$, which are zero for all but finitely
many pairs $(C',\cL')$.  By applying the functor $\Hom(\cdot,
\cIC(C',\cL'))$ to both sides, one sees that $V_{C',\cL'} \simeq
\Hom(K(C,\cL),\cIC(C',\cL'))^*$.  The desired formula for $K(C,\cL)$ then
follows from Lemma~\ref{lem:ext-std}. 
The second part of the theorem is proved similarly.
\end{proof}

From this structure theorem, we can deduce the following constraint on
when $\Ext^1$-groups may be nonzero.  This strengthens the
$\Ext^1$-vanishing result contained in~\cite[Proposition~11.2]{at:pd}.

\begin{cor}\label{cor:ext1-skdeg}
Suppose $\Ext^1(\cIC(C,\cL), \cIC(C',\cL')) \ne 0$.  Then either $C \subset \partial C'$ or $C' \subset \partial C$, and $\skdeg \cIC(C',\cL') = \skdeg \cIC(C,\cL) - 1$.
\end{cor}
\begin{proof}
By Corollary~\ref{cor:ic-semis}, we cannot have $C = C'$.  Next, if $C
\not\subset \partial C'$ and $C' \not\subset \partial C$, we established 
the vanishing statement~\eqref{eqn:ext-incomp} in the course of the proof
of Propostion~\ref{prop:full}.  Finally, if $C \subset \partial C'$, then
$\cIC(C',\cL')$ must occur as a direct summand of $K(C,\cL)$ by
Lemma~\ref{lem:ext-std}, and then its skew degree must be $\skdeg
\cIC(C,\cL) - 1$ according to Theorem~\ref{thm:std-struc}.  Similar
reasoning applies if $C' \subset \partial C$.
\end{proof}

Since the category $\cM(C)$ is semisimple, Theorem~\ref{thm:std-struc}
immediately implies the following more general statement about the
functors $\ru j_{C!}$ and $\ru j_{C*}$.

\begin{cor}\label{cor:std-struc}
For any object $\cF \in \cM(C)$, the kernel of the natural morphism $\ru j_{C!}\cF \to \ru j_{C!*}\cF$ is naturally isomorphic to
\[
\bigoplus_{\{(C',\cL') \in \orbvb X \mid C' \subset \partial C\}}
\Ext^1(\ru j_{C!*}\cF, \cIC(C',\cL'))^* \otimes \cIC(C',\cL').
\]
Likewise, the cokernel of the natural morphism $\ru j_{C!*}\cF \to \ru j_{C*}\cF$ is naturally isomorphic to 
\[
\bigoplus_{\{(C',\cL') \in \orbvb X \mid C' \subset \partial C\}}
\Ext^1(\cIC(C',\cL'), \ru j_{C!*}\cF) \otimes \cIC(C',\cL').\rlap{\qed}
\]
\end{cor}

\begin{thm}\label{thm:ext-finite}
For any $(C,\cL) \in \orbvb X$, there are only finitely many pairs
$(C',\cL') \in \orbvb X$ such that either 
\[
\Ext^1(\cIC(C,\cL), \cIC(C',\cL')) \ne 0
\qquad\text{or}\qquad
\Ext^1(\cIC(C',\cL'), \cIC(C,\cL)) \ne 0.
\]
\end{thm}
\begin{proof}
Since $X$ contains only finitely many orbits, and in view of Lemma~\ref{lem:ext-finite-below}, it suffices to prove that for a fixed orbit $C'$ with $C \subset \partial C'$, there are only finitely many irreducible vector bundles $\cL' \in \cg{C'}$ such that one of the $\Ext^1$-groups above is nonzero.  Moreover, by Proposition~\ref{prop:full}, those $\Ext^1$-groups may be computed in $\cM(\overline C{}')$ instead.  Thus, we henceforth assume without loss of generality that $X = \overline C{}'$.  Recall that $\cg X$ is assumed to have enough locally free objects.  Let us fix some locally free resolution $\cP^\bullet$ of $\cIC(C,\cL)$.  ($\cP^\bullet$ may be unbounded below, of course.)  Let
\[
m = d_\cL - \dim C + \dim C' - 1.
\]

Suppose now $(C',\cL')$ is such that
\begin{equation}\label{eqn:ext1-fail}
\Ext^1(\cIC(C,\cL), \cIC(C',\cL')) \ne 0.
\end{equation}
Corollary~\ref{cor:ext1-skdeg} tells us that $\skdeg \cIC(C',\cL') = \skdeg \cIC(C,\cL) - 1$, and it follows that $d_{\cL'} = m$.  In particular, $d_{\cL'}$ is independent of $\cL'$.  Now, choose some bounded complex $\cF^\bullet$ of $G$-equivariant coherent sheaves that represents $\cIC(C',\cL')$.  Then any nonzero class $f \in \Ext^1(\cIC(C,\cL), \cIC(C',\cL'))$ may represented by morphism of chain complexes $\tilde f: \cP^\bullet[-1] \to \cF^\bullet$.  Let $\cG^\bullet$ denote the image of $\tilde f$.

We claim that $h^m(\cG^\bullet)|_{C'} \simeq \cL'$.  Indeed, we clearly have $h^i(\cG^\bullet) \subset h^i(\cIC(C',\cL'))$ for all $i$, and in particular, $h^i(\cG^\bullet)|_{C'} = 0$ for $i \ne m$.  From the characterization of $\flatru \Dl X0$ in terms of cohomology sheaves in~\cite[Section~2]{at:bs}, it follows that the complex $\cG^\bullet$, regarded as an object of $\Db X$, also belongs to $\flatru \Dl X0$.  Moreover, if $h^i(\cG^\bullet)$ is supported on $\partial C'$ for all $i$, then we actually have $\cG^\bullet \in \ru \Dl X{-1}$.  But that is impossible: the natural map $\cG^\bullet \to \cIC(C',\cL')$ is nonzero (because $f \ne 0$), but there is no nonzero morphism from an object of $\ru \Dl X{-1}$ to an object of $\cM(X)$.  We noted earlier that $h^i(\cG^\bullet)$ is supported on $\partial C'$ for $i \ne m$, so we must have $h^m(\cG^\bullet)|_{C'} \ne 0$.  Indeed,
\[
h^m(\cG^\bullet)|_{C'} \subset h^m(\cIC(C',\cL'))|_{C'} \simeq \cL',
\]
and since $\cL'$ is irreducible, the containment is an equality.  Finally, note that $h^m(\cG^\bullet)|_{C'}$ is a subquotient of $\cP^m|_{C'}$. 

We have shown that the condition~\eqref{eqn:ext1-fail} implies that $\cL'$ is a subquotient of $\cP^m|_{C'}$.  Since $\cg{C'}$ is a finite-type category, there are, up to isomorphism, only finitely many possible irreducible vector bundles $\cL'$ for which~\eqref{eqn:ext1-fail} may hold, as desired.

It remains to consider $\cL'$ such that $\Ext^1(\cIC(C',\cL'), \cIC(C,\cL)) \ne 0$.  Unusually for this paper, the proof is not parallel to the case considered above.  Instead, we use the Serre--Grothendieck duality functor, which we denote $\D: \Db X \to \Db X$.  By~\cite[Theorem~8.6]{at:bs}, $\D$ carries the staggered $t$-structure to another staggered $t$-structure (that is, with respect to another perversity).  Since
\[
\Ext^1(\cIC(C',\cL'), \cIC(C,\cL)) \simeq \Ext^1(\D\cIC(C,\cL), \D\cIC(C',\cL')),
\]
the desired result follows from the case~\eqref{eqn:ext1-fail} in the dual category $\D\cM(X)$.
\end{proof}

\section{Projective and Injective Objects}
\label{sect:extfin}

In this section, we prove the main result of the paper:

\begin{thm}\label{thm:projinj}
Every simple object in $\cM(X)$ admits a canonical projective cover and a canonical injective hull.
\end{thm}

\begin{rmk}
For the reader familiar with~\cite[Theorem~3.2.1]{bgs}, on whose proof the
argument below is based, we briefly indicate the relationship between the
two.  That theorem states that in any quasi-hereditary category with
finitely many isomorphism classes of simple objects and a certain
$\Ext^2$-vanishing condition on standard and costandard objects, every
simple object has a projective cover and an injective hull.  
Both assumptions are false in general for staggered sheaves, so we cannot
simply invoke that theorem.  However, the fact that $\cM(X)$ usually has
infinitely many isomorphism classes of simple objects is not a major
obstacle: we proceed by induction on the number of orbits instead of on the
number of simple objects, making use of Theorems~\ref{thm:serre} and~\ref{thm:ext-finite} to handle infinitely many simple objects
at each step.  The failure of the $\Ext^2$-vanishing hypothesis is more
serious: getting around this requires a delicate argument using the
explicit structure theorem from Section~\ref{sect:struc} to gain control over what happens in the relevant $\Ext^2$-groups.  It is this aspect of the proof that causes it to be so lengthy.

In~\cite{bgs}, the $\Ext^2$-vanishing condition is also
used to show that projective covers admit standard filtrations, and
hence to deduce the reciprocity formula~\eqref{eqn:recip}.  However, these properties
need not hold for $\cM(X)$ (except in the weak sense of
Proposition~\ref{prop:std-filt}).  Indeed, they fail even in the example of
the $\bG_m$-action on $\bA^1$ considered in~\cite[Section~11.2]{a}.
\end{rmk}

\begin{rmk}
In the remarks following~\cite[Theorem~3.2.1]{bgs}, the authors sketch a different and much shorter existence proof for projective covers, due to Ringel.  That argument (which also assumes finitely many isomorphism classes of simple objects) could likely also be modified to work for $\cM(X)$.  However, the proof below has the advantage of giving a rather explicit description of those objects.
\end{rmk}

To prove Theorem~\ref{thm:projinj}, we proceed by induction on the number of orbits in $X$.  Choose an open orbit $C_0 \subset X$, and let $Z = X \ssm C_0$. 
Given a pair $(C,\cL) \in \orbvb X$, we will show that $\cIC(C,\cL)$ admits a projective cover.  The construction of injective hulls is parallel and will be omitted.  If $C = C_0$, then clearly $M(C,\cL)$ is the desired projective cover.  Assume henceforth that $C \subset Z$.  The treatment of this case occupies the remainder of the section.

\subsection{Definition of auxiliary objects}

In this step, we define eight new objects, arranged in seven interlocking short exact sequences.  A summary diagram is shown in Figure~\ref{fig:ses}.

\begin{figure}
\[
\xymatrix@=3pt{
&&&& R \ar[dddrr] && R \ar[dddrr]\ar[ddrrrr]
 &&&&&&&& P_Z \\
\\
&&&& && &&&& P \ar[uurrrr]\ar[drr] \\
&& K \ar[uuurr]\ar[drr]
 &&&& D \ar[dddrr] && D \ar[urr]\ar[dddrr]
 &&&& Q \ar[uuurr] \\
&&&& M \ar[urr]\ar[ddrrrr] \\
\\
J \ar[uuurr]\ar[uurrrr]
 &&&&&&&& S && S \ar[uuurr]}
\]
\caption{In this diagram, every collinear sequence of three objects is a short exact sequence.\label{fig:ses}}
\end{figure}
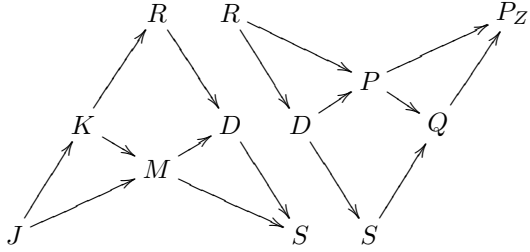

By assumption, within the category $\cM(Z)$, the object
$\cIC(C,\cL)$ has a projective cover, which we denote $P_Z$.  For each
irreducible vector bundle $\cL \in \cg {C_0}$, let
\[
B_\cL = \Ext^1(P_Z, \cIC(C_0,\cL)).
\]
It follows from Theorem~\ref{thm:ext-finite} that $B_\cL = 0$ for all but finitely many $\cL$ up to
isomorphism.  Thus, we may form the direct sum
\[
S = \bigoplus B_\cL^* \otimes \cIC(C_0,\cL).
\]
Note that for any irreducible vector bundle $\cL \in \cg {C_0}$, we have
\begin{multline}\label{eqn:shom}
\Hom(S, \cIC(C_0,\cL)) \simeq \Hom(B_\cL^* \otimes \cIC(C_0,\cL),
\cIC(C_0,\cL)) \\
\simeq \Hom(B_\cL^*, \Bbbk) \simeq B_\cL.
\end{multline}
This observation will be used later.  Consider now the sequence of isomorphisms
\[
\Ext^1(P_Z,S) \simeq \bigoplus B_\cL^* \otimes \Ext^1(P_Z, \cIC(C,\cL))
\simeq \bigoplus \End(B_\cL).
\]
Thus, $\Ext^1(P_Z,S)$ contains a canonical element $\alpha$, corresponding
to the identity operator in $\bigoplus \End(B_\cL)$.  We define an object
$Q$ by forming the short exact sequence corresponding to $\alpha$:
\begin{equation}\label{eqn:pses1}
0 \to S \to Q \to P_Z \to 0.
\end{equation}

Since $P_Z$ is projective as an object of subcategory $\cM(Z)$, and since $\cM(Z)$ is Serre by Theorem~\ref{thm:serre}, it follows
that even in the larger category $\cM(X)$, we have
\begin{equation}\label{eqn:pzext-p}
\Ext^1(P_Z, \cIC(C',\cL')) = 0 \qquad\text{when $C' \subset Z$.}  
\end{equation}
Thus, whenever $C' \subset Z$, we obtain from~\eqref{eqn:pses1} an exact
sequence
\begin{equation}\label{eqn:ples1}
0 \to \Ext^1(Q,\cIC(C',\cL')) \to \Ext^1(S,\cIC(C',\cL')) \to
\Ext^2(P_Z,\cIC(C',\cL')). 
\end{equation}
Let us put
\[
E_{C',\cL'} = \Ext^1(Q,\cIC(C',\cL')),
\qquad\text{and}\qquad
F_{C',\cL'} = \Ext^1(S, \cIC(C',\cL')).
\]
Next, we denote the cokernel of the inclusion $E_{C',\cL'} \hto
F_{C',\cL'}$ by
\begin{equation}\label{eqn:barf-inc}
\bar F_{C',\cL'} = F_{C',\cL'}/E_{C',\cL'} \subset
\Ext^2(P_Z,\cIC(C',\cL')).
\end{equation}

We now define a new object $R$ by a construction analogous to that of $S$:
\[
R = \bigoplus E_{C',\cL'}^* \otimes \cIC(C',\cL').
\]
As we saw with $\Ext^1(P_Z,S)$, the group $\Ext^1(Q,R) \simeq \bigoplus \End(E_{C',\cL'})$ contains a
canonical element $\beta$.  Let
\[
 0 \to R \to P \to Q \to 0
\]
be the corresponding short exact sequence.  Now, $S$ can be identified
with a subobject of $Q$.  Let $D$ be its preimage in $P$.  We thus obtain
an additional short exact sequence
\begin{equation}\label{eqn:pses3}
 0 \to R \to D \to S \to 0.
\end{equation}
The class $\gamma$ of this extension in $\Ext^1(S,R)$ can be described as
follows.  As in~\eqref{eqn:ples1}, the long exact sequence formed
from~\eqref{eqn:pses1} gives us an injective map $\Ext^1(Q,R) \hto
\Ext^1(S,R)$.  The element $\gamma$ is simply the image of $\beta$ under
this inclusion.  For an alternate description, note that
\[
\Ext^1(S,R) \simeq \bigoplus E_{C',\cL'}^* \otimes \Ext^1(S,\cIC(C',\cL'))
\simeq \bigoplus \Hom(E_{C',\cL'}, F_{C',\cL'}).
\]
Then $\beta$ corresponds to the inclusions $E_{C',\cL'} \to F_{C',\cL'}$.

It is clear that $P/D \simeq Q/S \simeq P_Z$, so we also have
\begin{equation}\label{eqn:pses4}
 0 \to D \to P \to P_Z \to 0.
\end{equation}

Next, consider the objects
\[
 M = \bigoplus B_\cL^* \otimes M(C_0,\cL)
\qquad\text{and}\qquad
K = \bigoplus F_{C',\cL'}^* \otimes \cIC(C',\cL').
\]
We of course have a surjective map $M \to S$.  According to
Corollary~\ref{cor:std-struc}, its kernel is $K$, so we have a short exact
sequence
\begin{equation}\label{eqn:pses5}
0 \to K \to M \to S \to 0.
\end{equation}
For each pair $(C',\cL')$, there is a natural surjective map
$F_{C',\cL'}^* \to E_{C',\cL'}^*$ with kernel $\bar F_{C',\cL'}^*$. 
Together, they give rise to a surjective map $K \to R$ with kernel
\[
J = \bigoplus \bar F_{C',\cL'}^* \otimes \cIC(C',\cL').
\]

To complete the picture in Figure~\ref{fig:ses}, it remains to show the existence of a
short exact sequence
\begin{equation}\label{eqn:pses7}
 0 \to J \to M \to D \to 0
\end{equation}
making the diagram commute.  There is a natural isomorphism $\Ext^1(S,K)
\simeq \End(F_{C',\cL'})$, and the canonical element $\delta \in
\Ext^1(S,K)$ corresponds to the extension~\eqref{eqn:pses5}.  Consider the following commutative diagram:
\[
\xymatrix@=12pt{
\Ext^1(S,K) \ar[r]\ar@{=}[d] & \Ext^1(S,R)\ar@{=}[d] & \Ext^1(Q,R)
\ar[l]\ar@{=}[d] \\
\End(F_{C',\cL'}) \ar[r] & \Hom(E_{C',\cL'},F_{C',\cL'}) &
\End(E_{C',\cL'}) \ar[l]}
\]
The image of $\delta$ in $\Ext^1(S,R)$ is clearly $\gamma$.  It follows
that the sequences~\eqref{eqn:pses3} and~\eqref{eqn:pses5} are related by
a commutative diagram as shown in Figure~\ref{fig:ses}, and, moreover, that the kernel
of the surjective morphism $M \to D$ coincides with that of $K \to R$, as
desired.

\subsection{Properties of $D$}

Given a pair $(C',\cL') \in \orbvb X$, form the exact sequence
\begin{multline}\label{eqn:ples7}
 0 \to \Hom(D, \cIC(C',\cL')) \to \Hom(M, \cIC(C',\cL')) \\
\to \Hom(J, \cIC(C',\cL')) \to \Ext^1(D,\cIC(C',\cL')) \to 0.
\end{multline}
Here, we have used the fact that $M$ is projective, and hence $\Ext^1(M,
\cIC(C',\cL')) = 0$.  Now, if $C' \subset Z$, then the second term
above vanishes, and it follows that
\begin{equation}\label{eqn:dhom-p}
\Hom(D,\cIC(C',\cL')) = 0
\end{equation}
and that
\[
\Ext^1(D,\cIC(C',\cL')) \simeq \Hom(J,\cIC(C',\cL')) \simeq \bar F_{C',\cL'}.
\]
In particular, the natural morphism $\Ext^1(D,\cIC(C',\cL')) \to \Ext^2(P_Z,\cIC(C',\cL'))$ in the long exact sequence for~\eqref{eqn:pses4} can be identified with the inclusion in~\eqref{eqn:barf-inc}.  That is, the morphism
\begin{equation}\label{eqn:dext-p}
\Ext^1(D,\cIC(C',\cL')) \hto \Ext^2(P_Z,\cIC(C',\cL'))
\end{equation}
is injective.
Next, consider the sequence~\eqref{eqn:ples7} in the case where $C' = C_0$.  In this case, the third term vanishes, so 
\begin{equation}\label{eqn:dext-b}
\Ext^1(D,\cIC(C_0,\cL')) = 0,
\end{equation}
and, moreover, the first two terms are isomorphic.  Since
$\Hom(K,\cIC(C_0,\cL')) = 0$, it follows from the long exact
sequence for~\eqref{eqn:pses5} that
\begin{multline}\label{eqn:dhom-b}
\Hom(D,\cIC(C_0,\cL')) \simeq \Hom(M,\cIC(C_0,\cL')) \\
\simeq \Hom(S,\cIC(C_0,\cL')) \simeq \Ext^1(P_Z,\cIC(C_0,\cL')),
\end{multline}
where the last isomorphism is from~\eqref{eqn:shom}.

\subsection{Conclusion of the proof}

Given $(C',\cL') \in \orbvb X$, form the sequence
\begin{multline*}
0 \to \Hom(P_Z,\cIC(C',\cL')) \to \Hom(P,\cIC(C',\cL')) \to \\
\Hom(D,\cIC(C',\cL')) \to \Ext^1(P_Z,\cIC(C',\cL'))
\to \Ext^1(P,\cIC(C',\cL')) \\
\to \Ext^1(D, \cIC(C',\cL')) \to \Ext^2(P_Z,\cIC(C',\cL')) \to \cdots.
\end{multline*}
Depending on whether $C' \subset Z$ or $C' = C_0$, we use either~\eqref{eqn:dhom-p} or~\eqref{eqn:dhom-b} to see that the first two terms are isomorphic:
\[
\Hom(P,\cIC(C',\cL')) \simeq \Hom(P_Z, \cIC(C',\cL')) \simeq
\begin{cases}
\Bbbk & \text{if $(C',\cL') = (C,\cL)$,} \\
0 & \text{otherwise.}
\end{cases}
\]
In other words, $\cIC(C,\cL)$ is the unique simple quotient of $P$.  A purity argument as in the proof of Theorem~\ref{thm:std-struc} then shows that $[P:\cIC(C,\cL)] = 1$.  Next, we check that $P$ is projective: if $C' \subset Z$, we see from~\eqref{eqn:pzext-p} and~\eqref{eqn:dext-p} that
\[
\Ext^1(P,\cIC(C',\cL')) = 0.
\]
If $C' = C_0$, the same result follows from~\eqref{eqn:dext-b} and~\eqref{eqn:dhom-b}.  Thus, $P$ is a projective cover of $\cIC(C,\cL)$, and it is unique up to canonical isomorphism by Lemma~\ref{lem:proj-iso}.

\section{An Example}
\label{sect:exam}

Let $G = (\bG_m)^2$ act on $\bA^2$ with weights $\pi_+ = (2,1)$ and $\pi_- = (-2,1)$.  That is, $G$ acts on $\bA^2$ by
\[
(z,t) \cdot \begin{bmatrix} u \\ v \end{bmatrix}
= \begin{bmatrix} z^2tu \\ z^{-2}tv \end{bmatrix}.
\]
Let $X = \Spec \Bbbk[x,y]/(xy)$; this is the union of the two coordinate axes in $\bA^2$.  Clearly, $G$ has three orbits on $X$, which we denote $C_+ = \Spec \Bbbk[x,x^{-1}]$, $C_- = \Spec \Bbbk[y,y^{-1}]$, and $C_0 = \Spec \Bbbk$.  The $G$-stabilizer of any point in $C_+$ is $H_+ = \ker \pi_+$, and the $G$-stabilizer of any point in $C_-$ is $H_- = \ker \pi_-$.  Note that $H_+$ is precisely the image of the cocharacter $\chi_+ = (-1,2)$ (that is, the map $\chi_+: \bG_m \to G$ given by $s \mapsto (s^{-1},s^2)$), and $H_-$ is the image of $\chi_- = (1,2)$.  The weight lattice of $H_+$ or $H_-$ can be identified with $\Z$, and the restriction of any $G$-weight $\lambda$ to $H_+$ or $H_-$ is given by $\la \chi_+,\lambda\ra$ or $\la\chi_-,\lambda\ra$, respectively.

\begin{rmk}
This example is related to the Lie group $SL(2,\R)$ in the following way.  Let $K$ denote a maximal compact subgroup of $SL(2,\R)$, and let $\fk \subset \fsl(2,\R)$ denote its Lie algebra.  Its complexification $K_\C \simeq \C^\times$ acts on the complex vector space $(\fsl(2,\R)/\fk)^* \otimes \C$ with weights $2$ and $-2$.  The variety $\cN$ of nilpotent elements in that space is the union of the two weight spaces.  Now, let us extend this $K_\C$-action to a $(K_C \times \C^\times)$-action by having the second factor act with weight $1$ on the whole vector space.  Then, the $(K_\C \times \C^\times)$-action on $\cN$ is isomorphic to the $G$-action on $X$ described above.
\end{rmk}

For any $G$-stable closed subscheme $Y \subset X$ and any $G$-weight $\lambda$, we write $\cO_Y(\lambda)$ for the $G$-equivariant sheaf obtained by twisting the structure sheaf of $Y$ by $\lambda$.  Next, note that irreducible line bundles on $C_+$ or on $C_-$ are indexed by $\Z$, {\it i.e.}, by characters of $H_+$ or $H_-$.  We denote these line bundles by $\cL_+(n)$ and $\cL_-(n)$.

Let $i: C_0 \hto X$ be the inclusion morphism.  It is easy to check that
\begin{equation}\label{eqn:oc-res}
\begin{aligned}
Li^*\cO_{\overline C_+}(\lambda) &\simeq \cO_{C_0}(\lambda), \\
Ri^!\cO_{\overline C_+}(\lambda) &\simeq \cO_{C_0}(\lambda+\pi_+)[-1],
\end{aligned}
\qquad\qquad
\cO_{\overline C_+}(\lambda)|_{C_+} \simeq \cL_+(\la \chi_+,\lambda\ra).
\end{equation}
Analogous results hold for $\cO_{\overline C_-}(\lambda)$.

Let us endow $X$ with an $s$-structure.  By~\cite[Theorem~7.4]{at:pd}, we may specify an
$s$-structure on an orbit by giving a cocharacter of its isotropy group. 
We give $C_+$ the $s$-structure corresponding to $\chi_+$ (which may, of
course, be regarded as a cocharacter of the isotropy group $H_+$), and we
give $C_-$ the $s$-structure corresponding to $\chi_-$.  For $C_0$, we use
the cocharacter $\chi_0 = (0,1)$.

To combine these into an $s$-structure on all of $X$, we will use the
gluing theorem~\cite[Theorem~1.1]{as:flag}.  That theorem requires us to
check that the conormal bundle of any orbit $C$ lies in $\cgl C{-1}$.  This
condition holds trivially for the open orbits $C_+$ and $C_-$.  For $C_0$,
the conormal bundle is $\cO_{C_0}(-\pi_+) \oplus \cO_{C_0}(-\pi_-)$.  We
have $\la\chi_0,-\pi_\pm\ra \le -1$, as required.  We thus obtain an
$s$-structure on $X$.  It is automatically recessed, and it is split
by~\cite[Theorem~7.6]{at:pd}, as required.

Let $r: \orb X \to \Z$ denote the constant perversity $r(C) = 0$.  We now
determine the simple staggered sheaves with respect to this perversity. 
For brevity, we adopt the
notation
\[
\begin{aligned}
 L_+(n) &= \cIC(\overline C_+, \cL_+(n)), \\
 L_-(n) &= \cIC(\overline C_-, \cL_-(n)),
\end{aligned}
\qquad\qquad
 L_0(n,k) = \cIC(C_0, \cO_{C_0}(n,k)).
\]
The three families of simple objects can be explicitly described as
follows:
\[
\begin{aligned}
L_+(n) &= \cO_{\overline C_+}(n-2,n-1)[n], \\
L_-(n) &= \cO_{\overline C_-}(-n+2,n-1)[n],
\end{aligned}
\qquad\qquad
L_0(n,k) = \cO_{C_0}(n,k)[k].
\]
(These assertions are easily verified using~\eqref{eqn:oc-res}.)  Next, we turn to standard and costandard objects.  It turns that every simple object is also standard, and of course the $L_0(n,k)$ are also costandard.  The nontrivial costandard objects are
\[
N_+(n) = \cO_{\overline C_+}(n,n)[n]
\qquad\text{and}\qquad
N_-(n) = \cO_{\overline C_-}(-n,n)[n].
\]
These objects fit into short exact sequences as follows:
\begin{equation}\label{eqn:oc-ses}
\begin{gathered}
0 \to L_+(n) \to N_+(n) \to L_0(n,n) \to 0 \\
0 \to L_-(n) \to N_-(n) \to L_0(-n,n) \to 0
\end{gathered}
\end{equation}

Finally, we turn to projective covers and injective hulls.  The objects $L_\pm(n)$ are already projective, and their injective hulls are simply their costandard hulls, the $N_\pm(n)$.  On the other hand, the objects $L_0(n,k)$ are already injective, and they are also projective except when $k = |n|$.  In that case, the nonsplit sequences~\eqref{eqn:oc-ses} show that they are not projective.  In fact, when $n \ne 0$, we can read off the projective cover $P_0(\pm n,n)$ of $L_0(\pm n,n)$ from~\eqref{eqn:oc-ses}:
\[
P_0(n,n) = N_+(n)
\qquad\text{and}\qquad
P_0(-n,n) = N_-(n)\qquad\text{if $n \ne 0$.}
\]
In the special case $n = 0$, the sequences~\eqref{eqn:oc-ses} give two distinct nontrivial extensions of $L_0(0,0)$.  The projective cover of this object is $P_0(0,0) \simeq \cO_X$, and we have a short exact sequence
\[
0 \to L_+(0) \oplus L_-(0) \to P_0(0,0) \to L_0(0,0) \to 0.
\]

Because every standard object in this example happens to be simple, it is
obviously true that projective covers of simple objects have
standard filtrations.  As we saw in Section~\ref{sect:abcat}, the
multiplicities of standard objects in projective covers obey the
reciprocity formula~\eqref{eqn:recip}.


\end{document}